\documentclass{article}
\usepackage{amsthm, amsmath, amssymb, amsfonts, enumerate, vmargin, accents,
}
\usepackage[usenames,dvipsnames]{color}
\usepackage{hyperref}
\usepackage{tikz}
\usepackage{multicol}
\usetikzlibrary{decorations.markings}
\usetikzlibrary{calc}
\usetikzlibrary{arrows}

\usepackage{epstopdf}

\setmarginsrb{30mm}{20mm}{30mm}{20mm}{0pt}{10mm}{0pt}{10mm}

\def\N{\mathbb{N}}

\def\Z{\mathbb{Z}}
\def\R{\mathbb{R}}

\def\Hiota{\hat{\iota}}
\def\HHiota{\hat{\vphantom{\rule{1pt}{5.5pt}}\smash{\hat{\iota}}}}

\newtheorem{thm}{Theorem}
\newtheorem*{thm*}{Theorem}

\newtheorem*{claim*}{Claim}

\newtheorem*{dfn*}{Definition}
\newtheorem{lemma}[thm]{Lemma}
\newtheorem*{lemma*}{Lemma}
\newtheorem{prop}[thm]{Proposition}
\newtheorem*{prop*}{Proposition}
\newtheorem{cor}[thm]{Corollary}
\newtheorem*{cor*}{Corollary}

\newtheorem*{conj*}{Conjecture}

\newtheorem*{quest*}{Question}

\newtheorem*{exer*}{Exercise}
\theoremstyle{remark}

\newtheorem*{rmk*}{Remark}

\newtheorem*{example*}{Example}

\newtheorem*{examples*}{Examples}

\newtheorem*{rmks*}{Remarks}

\title{
The Writhe of Permutations \\
and Random Framed Knots
}

\author{Chaim Even-Zohar\thanks{
Department of Mathematics, Hebrew University, Jerusalem 91904, Israel. \newline
e-mail: \href{mailto:chaim.evenzohar@mail.huji.ac.il}{chaim.evenzohar@mail.huji.ac.il}~.}}

\begin{document}

\maketitle

\begin{abstract}
We introduce and study the \emph{writhe} of a permutation, a circular variant of the well-known inversion number. This simple permutation statistics has several interpretations, which lead to some interesting properties. For a permutation sampled uniformly at random, we study the asymptotics of the writhe, and obtain a non-Gaussian limit distribution.

This work is motivated by the study of random knots. A model for random framed knots is described, which refines the Petaluma model, studied in~\cite{petaluma}. The distribution of the framing in this model is equivalent to the writhe of random permutations.
\end{abstract}

\section{Introduction}

As usual we denote $\Z_{2n+1} = \Z / (2n+1)\Z$. Let $f: \Z_{2n+1}\to\mathbb R$ be one-to-one. We define the \emph{writhe} of $f$ as follows.
$$ w(f) \;=\; \sum\limits_{i=0}^{2n} \sum\limits_{j=1}^{n} \text{sign}\left(f(i+j)-f(i)\right) $$
where $\text{sign}(x)=+1$ if $x>0$ and $-1$ if $x<0$.

Let's examine this definition. We sum over all pairs $x,y \in \Z_{2n+1}$ and compare the values of $f(x)$ and $f(y)$. Viewing $\Z_{2n+1}$ as a set of points evenly distributed on the unit circle, say clockwise, we may ``connect'' $x$ and $y$ through the shorter arc. Then we add one if $f$ is clockwise increasing on the pair $(x,y)$, and subtract one if it's decreasing. Intuitively, the writhe measures to what extent a function on the discrete circle is clockwise increasing or decreasing.

Note that in this definition only the relative ordering of the values in $\text{Image}(f)$ matters, and not the actual values. Hence it is natural to study the writhe for all one-to-one functions \mbox{$f:\{0,\dots,2n\} \to \{0,\dots,2n\}$}. The main subject of this paper is the distribution of the writhe for uniformly sampled permutations in $S_{2n+1}$.

Both the term writhe and our interest in the subject originated in knot theory. In~\cite{petaluma}, we study a model for random knots, in which a permutation $\pi \in S_{2n+1}$ defines a knot, through a petal diagram, from~\cite{adams2015knot}. In Section~\ref{section-knots}, we refine this random model to deal with \emph{framed} knots, and observe that the knot's \emph{framing} is equal to the writhe of $\pi$.

However, writhe seems to carry an independent interest as a combinatorial concept. Note its similarity with the inversion number of a permutation, $\text{inv}(\pi) = \#\{(x,y) : x<y,\;\pi(x)>\pi(y)\}$, which plays a central role in enumerative combinatorics as well as in sorting and ranking algorithms. In a sense, $\text{inv}(\pi)$ measures the distance of $\pi$ from the identity. For example, the Mallows model is a well-known non-uniform distribution of permutations, in which the probability of a permutation $\pi$ is proportional with $q^{\text{inv}(\pi)}$, for some $q>0$~\cite{mallows1957non,diaconis2000analysis,borodin2010adding}. 

In a similar spirit, the writhe quantifies how shuffled a given permutation $\pi:\Z_{2n+1}\to\Z_{2n+1}$ is. It is easily verified that the writhe is invariant under rotations from both sides. Namely, $w(\pi) = w(\pi \circ \rho) = w(\rho \circ \pi)$ where $\rho(x) = x + 1\mod (2n+1)$. Indeed the appropriate geometry for the writhe is that of the circle $\Z_{2n+1}$, whereas $\text{inv}(\pi)$ suits a linear order. 

More generally, it is interesting to study the distribution of \emph{graphical inversion numbers}, introduced by Foata and Zeilberger~\cite{foata1996graphical}. A permutation statistic $\text{inv}_G$ is associated to a directed graph $G$ with vertices~$\{1,\dots,k\}$. For $\pi \in S_k$, $\text{inv}_G(\pi)$ is the number of edges $u \to v$ such that $\pi(u)>\pi(v)$. Here are a few examples.
\begin{itemize}
\item 
For a complete transitive graph $G$, $\text{inv}_G$ is the usual inversion number. It follows the \mbox{\emph{Mahonian}} distribution which is asymptotically normal, as described in Section~\ref{section-properties-1}.
\item 
An oriented path or cycle corresponds to the permutation \emph{descent number}. Its distribution is given by the \emph{Eulerian numbers}, cf.~Section~\ref{section-moments-2}, with normal asymptotics too~\cite{carlitz1972asymptotic}.
\item 
For a complete bipartite $G=(U,V,E)$, $E=\left\{u \to v \;:\; u \in U,\, v \in V \right\}$, or $U \rightrightarrows V$ for short, $\text{inv}_G$ is the Mann--Whitney statistic, asymptotically normal as well~\cite{mann1947test}.
\item 
Bisect $U$ and $V$, and consider $U_1 \rightrightarrows V_1 \rightrightarrows U_2 \rightrightarrows V_2 \rightrightarrows U_1$. Now $\text{inv}_G$ is asymptotically \emph{logistic} with density $f(x) = \tfrac12 \,\text{sech}\,^2 x$. See~\cite{petaluma} and Proposition~\ref{compwrithe} below.
\item 
Finally, the writhe $w(\pi)$ for $\pi \in S_{2n+1}$ is equal to $(2n+1)n - 2\, \text{inv}_G(\pi)$, where $G$ is the \emph{clockwise tournament} on $\Z_{2n+1}$ with edge set $\left\{i \to i+j \;:\; i \in \Z_{2n+1},\, 1 \leq j \leq n \right\}$. 
\end{itemize}
A~related notion of \emph{weighted inversion numbers} was considered by Kadell~\cite{kadell1985weighted}.

In Section~\ref{section-properties} we discuss the relation of the writhe to other combinatorial notions, such as the \emph{alternating inversion number}~\cite{chebikin2008variations}, and \emph{circular bubble sort}~\cite{jerrum1985complexity,van2014upper}. We also fit the writhe into the context of nonparametric statistics for \emph{circular rank correlation}~\cite{fisher1982nonparametric,fisher1995statistical}. An algorithm to compute the writhe in time $O(n\log n)$ is described, inspired by notions from knot theory. 

\bigskip

The main result of this work is the distribution of the writhe of a uniformly chosen random permutation from $S_{2n+1}$. Although the writhe takes values in $[-n^2, n^2]$, we show that typically it is linear in $n$, in a strong sense which we turn to describe. 

We think of the writhe as a random variable and consider its moments $E\left[w^k\right]$. All odd moments vanish, since reversing $\pi$ flips the writhe's sign. We determine the asymptotics of the $k$th moment for even $k$. 
\begin{thm}\label{moments}
Let $\pi_{2n+1} \in S_{2n+1}$ be drawn uniformly at random. Then
$$ E\left[\frac{w(\pi_{2n+1})^k}{n^k}\right] \;\;\xrightarrow[]{\;n \to \infty\;}\;\; \sum_{\sigma \in S_k} \;\prod_{\gamma \in \text{cycles}(\sigma)} \lambda_{|\gamma|}\;, $$
where for even $m$
$$ \lambda_m \;=\; \frac{8^m(2^m-1)B_m^2}{2(m!)^2}\;, $$
and $\lambda_m=0$ for odd $m$. $B_m$ stands for the $m$th Bernoulli number.
\end{thm}
These limits are proved in Section~\ref{section-moments}, where we also give a formula for lower order terms. Our methods extend previous work on invariants of random knots~\cite{petaluma}, and work by Mingo and Nica~\cite{mingo1998distribution}.

Define the \emph{normalized writhe} of a random permutation $\pi \in S_{2n+1}$, as $W_{2n+1} = w(\pi)/n$. Theorem~\ref{moments} yields the limits of the moments of the normalized writhe, e.g.
$$ E\left[W_{2n+1}^2\right] \;\xrightarrow[]{\;n \to \infty\;}\; \frac{2}{3} 
\;\;\;\;\;\;\;\;\;\; \;\;\;\;\;\;\;\;\;\; 
E\left[W_{2n+1}^4\right] \;\xrightarrow[]{\;n \to \infty\;}\; \frac{76}{45} $$
We use the method of moments to derive from Theorem~\ref{moments} convergence in distribution of the normalized writhe.
\begin{cor}\label{limdist}
There exists a continuous distribution $W$ on $\R$ such that 
$$ W_{2n+1} \;\;\xrightarrow[\;n \to \infty\;]{D}\;\; W \;,$$
and $W$ is defined by the moment generating function 
$$\mathcal{M}(z) \;=\; E\left[e^{Wz}\right] \;=\; \exp\left(\sum\limits_{m=0}^{\infty}\frac{\lambda_m}{m}z^m\right)\;.$$
where $\lambda_m$ is as in Theorem~\ref{moments}.
\end{cor}

See Figure~\ref{writhe-hist} for the probability density function of $W$. In Section~\ref{section-limit} we prove Corollary~\ref{limdist}, and investigate some properties of $W$. It turns out that the tails of the distribution of the normalized writhe decay exponentially, at the asymptotic rate given below.
$$ \lim_{t \to \infty} \frac{\log P\left[|W|>t\right]}{t} \;=\; -\frac{\pi^2}{4} \;\approx\; -2.47 \;. $$ 

We present some alternative characterizations of the limit distribution. The infinite series in Corollary~\ref{limdist} may be traded for a definite complex integral, which is stated in Section~\ref{section-limit} as well. We also represent $W$ as an infinite sum. Let $A$ be a random variable with density $f_A(x) = \tfrac{1}{\pi} \,\text{sech}\, x$, and let $A_n \sim A$ iid. Then the infinite sum $\frac{4}{\pi^2} \sum_{n=1}^\infty \frac{1}{n} A_n$ converges almost surely to a random variable having the same distribution as~$W$. Refinements of this series yield further representations of $W$, using Laplace random variables, or products of Gaussians.

We note that none of these properties provides an efficient way to sample from $W$ or compute its percentiles. It is still desirable to find a closed form formula for that purpose. 

\begin{figure}[h]
\begin{center}
\includegraphics[width=\textwidth]{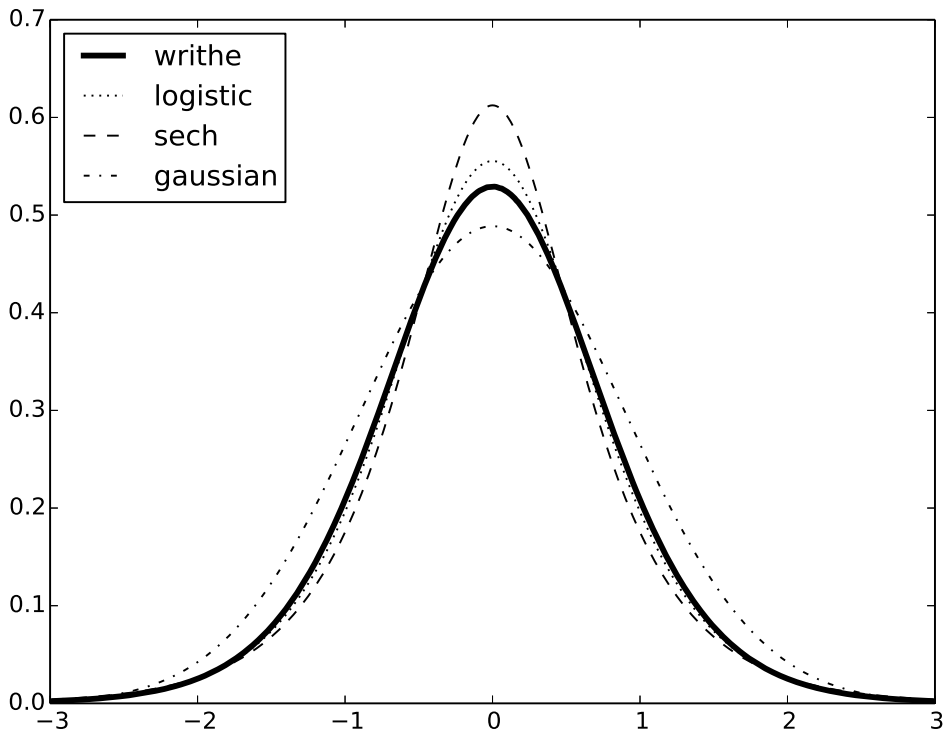}
\end{center}
\caption{The writhe distribution with friends. The probability density function of $W$, the limit distribution of the normalized writhe is given by the solid line. The Gaussian, Logistic, and Sech distributions with variance $2/3$ are given for comparison.}
\label{writhe-hist}
\end{figure}

\subsection*{Acknowledgment}\label{ack}

I wish to thank Joel Hass, Nati Linial, and Tahl Nowik for many hours of insightful discussions on random knots in general, and on the subject of this work. I also wish to thank Yuval Peres for a helpful discussion, and the anonymous referee for useful comments and suggestions. This project was supported by BSF grant 2012188. 

\section{Random Framed Knots}\label{section-knots}

We begin with describing the relation to random framed knots. More background on framed knots can be found, e.g., in~\cite{chmutov2012introduction}, and for further information on the random model, see~\cite{adams2015knot,petaluma}. Readers interested mainly in the combinatorial aspect may skip this section.

\subsection{Framed Knots}

A \emph{knot} is a simple closed curve in $3$ dimensions, or more precisely, a smooth embedding $S^1 \hookrightarrow \R^3$ up to the following equivalence. Two such curves represent the same knot iff one can be deformed into the other through an isotopy of the ambient space. 

\begin{figure}[htb]
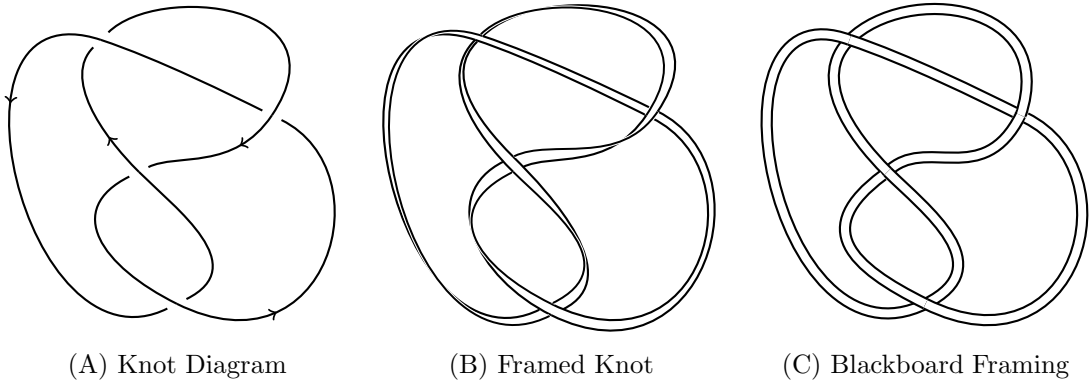

\begin{center}
\begin{tabular}{ccc}
\tikz[thick,decoration={markings,mark=at position 0.3 with {\arrow{>}}}]{
\draw[shorten <=4pt] (-0.5,0) .. controls +(-1,-0.5) and +(-1,0.5) .. (0,-1.75);
\draw[shorten >=4pt,postaction={decorate}] (0,-1.75) .. controls +(2,-1) and +(1.5,-0.75) .. (1.25,0.75);
\draw[shorten <=4pt] (1.25,0.75) .. controls +(-0.5,0.25) and +(0.75,-0.25) .. (-1,1.75);
\draw[shorten >=4pt,postaction={decorate}] (-1,1.75) .. controls +(-2.25,0.75) and +(-2,-1) .. (0,-1.75);
\draw[shorten <=4pt] (0,-1.75) .. controls +(1,0.5) and +(0.5,-0.5) .. (-0.5,0);
\draw[shorten >=4pt,postaction={decorate}] (-0.5,0) .. controls +(-0.5,0.5) and +(-0.5,-0.5) .. (-1,1.75);
\draw[shorten <=4pt] (-1,1.75)  .. controls +(0.75,0.75) and +(1,1.5) .. (1.25,0.75);
\draw[shorten >=4pt,postaction={decorate}] (1.25,0.75) .. controls +(-0.5,-0.75) and +(0.5,0.25) .. (-0.5,0);
\pgfresetboundingbox \clip (-2.25,-2.25) rectangle (2.25,2.25);} &
\tikz[thick,decoration={markings,mark=at position 0.3 with {\arrow{>}}}]{
\draw[shorten <=4pt] (-0.5,0.2) .. controls +(-1,-0.5) and +(-1,0.5) .. (0,-1.9);
\draw[shorten <=0.4pt,double distance=0.75pt,white,double=black,thin] (-0.5,0) .. controls +(-1,-0.5) and +(-1,0.5) .. (0,-1.75);
\draw[shorten >=0pt] (0,-1.9) .. controls +(2,-1) and +(1.5,-0.75) .. (1.4,0.8);
\draw[shorten >=2.5pt] (0,-1.75) .. controls +(2,-1) and +(1.5,-0.75) .. (1.25,0.75);
\draw[shorten <=2.5pt] (1.4,0.8) .. controls +(-0.5,0.25) and +(0.75,-0.25) .. (-0.9,1.8);
\draw[shorten <=0pt] (1.25,0.75) .. controls +(-0.5,0.25) and +(0.75,-0.25) .. (-1,1.75);
\draw[shorten >=5pt] (-1,1.75) .. controls +(-2.25,0.75) and +(-2,-1) .. (0,-1.75);
\draw[shorten >=0.7pt,double distance=0.75pt,white,double=black,thin] (-0.9,1.8) .. controls +(-2.25,0.75) and +(-2,-1) .. (0,-1.9);
\draw[shorten <=5pt] (0,-1.9) .. controls +(1,0.5) and +(0.5,-0.5) .. (-0.5,0.2);
\draw[shorten <=0.7pt,double distance=0.75pt,white,double=black,thin] (0,-1.75) .. controls +(1,0.5) and +(0.5,-0.5) .. (-0.5,0);
\draw[shorten >=3pt] (-0.5,0.2) .. controls +(-0.5,0.5) and +(-0.5,-0.5) .. (-0.9,1.8);
\draw[shorten >=0pt] (-0.5,0) .. controls +(-0.5,0.5) and +(-0.5,-0.5) .. (-1,1.75);
\draw[shorten <=0pt] (-0.9,1.8)  .. controls +(0.75,0.75) and +(1,1.5) .. (1.4,0.8);
\draw[shorten <=3pt,double distance=0.75pt,white,double=black,thin] (-1,1.75)  .. controls +(0.75,0.75) and +(1,1.5) .. (1.25,0.75);
\draw[shorten >=4pt] (1.25,0.75) .. controls +(-0.5,-0.75) and +(0.5,0.25) .. (-0.5,0);
\draw[shorten >=0.4pt,double distance=0.75pt,white,double=black,thin] (1.4,0.8) .. controls +(-0.5,-0.75) and +(0.5,0.25) .. (-0.5,0.2);
\pgfresetboundingbox \clip (-2.25,-2.25) rectangle (2.25,2.25);} &
\tikz[thick,decoration={markings,mark=at position 0.3 with {\arrow{>}}}]{
\draw[double distance=3pt,shorten <=0pt] (-1,1.75)  .. controls +(0.75,0.75) and +(0.5,1.5) .. (1.25,0.75);
\draw[double distance=3pt,shorten >=0pt] (1.25,0.75) .. controls +(-0.5,-1) and +(0.5,0.5) .. (-0.5,0);
\draw[double distance=3pt,shorten >=0pt] (-0.5,0) .. controls +(-0.5,0.5) and +(-0.5,-0.5) .. (-1,1.75);
\draw[double distance=3pt,shorten >=0pt] (-1,1.75) .. controls +(-2,0.75) and +(-2,-1) .. (0,-1.75);
\draw[double distance=3pt,shorten <=0pt] (0,-1.75) .. controls +(1,0.5) and +(0.5,-0.5) .. (-0.5,0);
\draw[double distance=3pt,shorten <=2.2pt] (-0.5,0) .. controls +(-1,-0.75) and +(-1,0.5) .. (0,-1.75);
\draw[double distance=3pt,shorten >=2.2pt] (0,-1.75) .. controls +(2,-1) and +(1.5,-0.75) .. (1.25,0.75);
\draw[double distance=3pt,shorten <=2.4pt] (1.25,0.75) .. controls +(-0.5,0.25) and +(0.75,-0.25) .. (-1,1.75);
\pgfresetboundingbox \clip (-2.25,-2.25) rectangle (2.25,2.25);} \\
(A) Knot Diagram & (B) Framed Knot & (C) Blackboard Framing
\end{tabular}
\end{center}
\caption{The figure eight knot, with two framings.}
\label{fig-knots}
\end{figure}

A knot can be represented by a \emph{knot diagram}, which is its projection to the plane, with a finite number of double points, called \emph{crossings}, marked to distinguish the over-strand from the under-strand. See Figure~\ref{fig-knots}A for a diagram of the \emph{figure-eight} knot.

A \emph{framed knot} is equipped with a normal vector, defined continuously at each point of the curve. By pushing the curve a bit along the framing vectors, one obtains an embedding of an orientable closed ribbon $[0,\varepsilon] \times S^1$, as illustrated in Figure~\ref{fig-knots}B. 

A natural way to choose a framing given a knot diagram, is parallel to the plane of the projection. This choice is called the \emph{blackboard framing}, since it can be obtained by drawing two parallel copies of the knot diagram, as in Figure~\ref{fig-knots}C.

\def\positive{\tikz[thick,->,baseline=-2]{\draw[black](0.3,-0.13) -- (0,0.17);\draw[white,-,line width=3](0.03,-0.13) -- (0.33,0.17);\draw[black](0.03,-0.13) -- (0.33,0.17);}}
\def\negative{\tikz[thick,->,baseline=-2]{\draw[black](0.03,-0.13) -- (0.33,0.17);\draw[white,-,line width=3](0.3,-0.13) -- (0,0.17);\draw[black](0.3,-0.13) -- (0,0.17);}}

\subsection{Writhe and Framing}

There are two types of crossing points in knot diagrams: positive~\positive, and negative~\negative. We first define the \emph{writhe of a knot diagram} $D$, as a signed sum over crossings in the diagram. Namely, 
$$w(D) = \#\positive-\#\negative$$
where the count is over all crossings in $D$. For example, the writhe of the diagram in Figure~\ref{fig-knots}A is $0$, as the two points in the lower part of the diagram are positive, and the other two are negative. The writhe cannot be pulled back to an isotopy invariant of knots, since it is not preserved under the following local change of diagrams, 
$ 
\tikz[thick,baseline=-2,scale=0.5]{
\draw[-,black](-0.5,-0.1) .. controls +(0:0.2) and +(180:0.2) .. (0,0);
\draw[-,black](0,0) .. controls +(0:0.2) and +(180:0.2) .. (0.5,-0.1);}
\Leftrightarrow 
\tikz[thick,baseline=-2,scale=0.5]{
\draw[-,black](-0.5,-0.1) .. controls +(0:0.2) and +(225:0.2) .. (0,0);
\draw[-,black](0,0) .. controls +(45:0.2) and +(0:0.2) .. (0,0.4);
\draw[-,white,line width=2](0,0.4) .. controls +(180:0.2) and +(135:0.2) .. (0,0);
\draw[-,white,line width=2](0,0) .. controls +(-45:0.2) and +(0:0.2) .. (0.5,-0.1);
\draw[-,black](0,0.4) .. controls +(180:0.2) and +(135:0.2) .. (0,0);
\draw[-,black](0,0) .. controls +(-45:0.2) and +(180:0.2) .. (0.5,-0.1);}
$, known as the Reidemeister move of type~I. A well-known use of the writhe appears in the definition of the Jones Polynomial through the Kauffman Brackets~\cite{kauffman1987knots}.

A \emph{link} is a collection of knots, or more precisely, a smooth embedding of a disjoint union of copies of $S^1$ into $\R^3$, as usual considered up to isotopies. Let $L$ be a two-component link, and let $D$ be an arbitrary link diagram of $L$. The \emph{linking number} is a well-known link invariant, firstly defined by Gauss~\cite{gauss1841general}. Here we use the definition
$$ lk(L) = \tfrac12\left(\#\positive-\#\negative\right) $$
where the count is only over crossings in $D$ that involve one strand from each component. One can show that the linking number doesn't depend on the choice of the diagram $D$. Hence it is well-defined as a link invariant. For example, 
$lk(\;
\tikz[thick,baseline=-3,scale=0.5,decoration={markings,mark=at position 0.05 with {\arrow{<}}}]{
\draw[-,gray](0.6,0) .. controls +(-90:0.4) and +(-90:0.4) .. (0,0);
\draw[-,white,line width=2](1,0) .. controls +(-90:0.4) and +(-90:0.4) .. (0.4,0);
\draw[-,black,postaction={decorate}](1,0) .. controls +(-90:0.4) and +(-90:0.4) .. (0.4,0);
\draw[-,black](0.4,0) .. controls +(90:0.4) and +(90:0.4) .. (1,0);
\draw[-,white,line width=2](0,0) .. controls +(90:0.4) and +(90:0.4) .. (0.6,0);
\draw[-,gray,postaction={decorate}](0,0) .. controls +(90:0.4) and +(90:0.4) .. (0.6,0);}
\;)=1$, as both crossings are positive and involve the two components.

The \emph{framing number} of a framed knot, also known as its \emph{self-linking}, is defined as the linking number of the knot, with the parallel curve obtained by small pushing in the direction of the framing vector, cf. Figure~\ref{fig-knots}B. The framing number is $0$ exactly when the ribbon described above can be extended to an embedded oriented surface whose boundary is the original knot. Since it is preserved under continuous deformations, the framing number is an invariant of framed knots. In fact, a framed knot is characterized by its framing number and the underlying knot type.

By the above definition, the framing number of the blackboard framing of a given knot diagram is equal to the writhe of that diagram, and hence can be computed as a sum of crossing signs. Indeed, each crossing in the knot diagram corresponds to four crossings in the link obtained from its blackboard framing. Two of these crossings involve the different components, the sign of both is equal to that of the original crossing. This is balanced by the coefficient $\tfrac12$ in the formula for the linking number.

The reader should be warned, however, that these two notions do not necessarily coincide for projections of framed knots that are not given by blackboard framing. For example, while both in Figures~\ref{fig-knots}B and~\ref{fig-knots}C the writhe of the diagram is $0$, the framing numbers of the knots are different, specifically $1$ and $0$ respectively. In general, the difference between the framing number of a framed knot and the writhe of its projection counts, in some appropriate sense, how many times the framing vector twists around the projected knot~\cite{adams1994knot}.

We mention two basic properties of the writhe that follow easily from the definition. Changing the orientation of the diagram preserves its writhe, while passing to a mirror image or flipping all crossings negates it.

\newcommand{\AD}[1]
{\begin{tikzpicture}[baseline=-3pt,>=stealth]
\draw[thick] (0,0) circle[radius=0.25]; \filldraw (0,0.25) circle[radius=0.05];
\foreach \start/\end/\sign/\cosign in {#1} \draw[->] (\start:0.25) -- node[pos=-0.2,scale=0.6]{$\sign$} node[pos=1.2,scale=0.6]{$\cosign$} (\end:0.25);
\clip (-0.4,-0.4) rectangle (0.4,0.4);
\end{tikzpicture}}

\begin{rmks*} See~\cite{chmutov2012introduction} or~\cite{petaluma} for definitions and more details on the following notions.
\begin{enumerate}
\item
A \emph{Gauss diagram formula} for the writhe is given by
$$ w(D) \;\;=\;\; \left\langle \AD{210/30//} +\AD{330/150//}\;,\;D\right\rangle \;\;=\;\; \left\langle \AD{210/30/+/} - \AD{210/30/-/} + \AD{330/150/+/} - \AD{330/150/-/}\;,\;D\right\rangle $$ 
\item 
The framing number is a \emph{first order} \emph{finite type invatiant} of framed knots. 
\end{enumerate}
\end{rmks*}

\subsection{The Random Model}

A \emph{petal diagram} is a closed planar curve, comprised of 2n + 1 straight segments crossing at a single point, and arcs connecting consecutive pairs of segment tips. See Figure~\ref{fig-petals}A for a petal diagram with $7$ petals. Adams et al. showed that this representation is universal, i.e., every knot can be positioned so that it has a petal projection~\cite{adams2015knot}. 

\begin{figure}[htb]
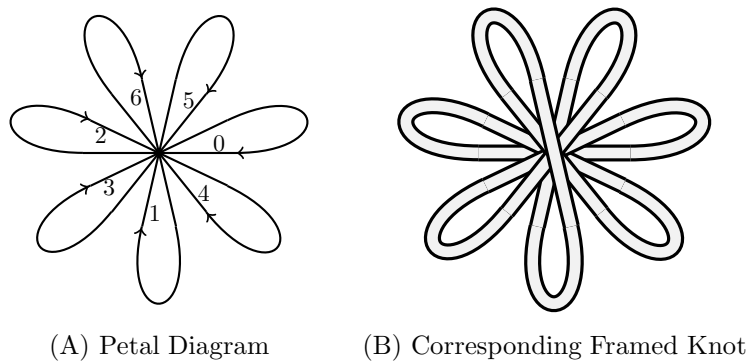

\begin{center}
\begin{tabular}{cc}
\tikz[thick,decoration={markings,mark=at position 0.05 with {\arrow{<}}}]{
\foreach \angle in {0, 51.4, ..., 320} 
\draw[postaction={decorate}] (\angle:1) .. controls +(\angle:1.4) and +(\angle+25.7:1.4) .. (\angle+25.7:1);
\foreach \angle/\num in {0/0,  51.4/4, 102.8/1, 154.2/3, 205.6/2, 257.0/6, 308.4/5} {
\node[scale=0.9] at (-\angle+9:0.8) {$\num$};
\draw (-\angle:1) -- (-\angle:-1);}
\pgfresetboundingbox \clip (-2.25,-2.25) rectangle (2.25,2.25);} &
\tikz[thick,decoration={markings,mark=at position 0.05 with {\arrow{<}}}]{
\foreach \angle in {0, 51.4, ..., 320} 
\draw[very thick,double distance=4pt,double=black!5] (\angle:1) .. controls +(\angle:1.4) and +(\angle+25.7:1.4) .. (\angle+25.7:1);
\foreach \angle/\num in {0/0,  102.8/1, 205.6/2, 154.2/3, 51.4/4, 308.4/5, 257.0/6} {
\draw[very thick,double distance=4pt,double=black!5] (-\angle:1) -- (-\angle:-1);}
\pgfresetboundingbox \clip (-2.25,-2.25) rectangle (2.25,2.25);} \\
(A) Petal Diagram & (B) Corresponding Framed Knot
\end{tabular}
\end{center}
\caption{A diagram with $7$ petals, and its blackboard framing. }
\label{fig-petals}
\end{figure}

Consider a petal diagram with $2n+1$ petals. The only variable is the relative ordering of the heights of the arcs in the original knot, as they pass straightly above the center point. This ordering is encoded by a permutation $\sigma \in S_{2n+1}$, that assigns a height to each arc, as marked in Figure~\ref{fig-petals}A. In the \emph{Petaluma} model, we generate a random knot, simply by picking $\sigma$ uniformly at random~\cite{petaluma}.

There is a natural extension of this random model to framed knots. We assign the blackboard framing to a petal diagram, with random $\sigma \in S_{2n+1}$ as in the Petaluma model, to obtain a random framed knot. This model is universal for framed knots --  see Proposition~\ref{universal} below. See also Figure~\ref{fig-petals}B for an illustration of a random framed knot. In the rest of this section, we will draw knot diagrams without the double lines that indicate the use of the blackboard framing.

What is the writhe of a petal knot diagram? Note that petal diagrams are not a special case of regular knot diagrams, because here the center point has more than two pre-images. However, the arcs above the center can be slightly perturbed without changing the knot type so as to obtain a decent knot diagram with $\tbinom{2n+1}{2}$ crossings. These crossings correspond to all pairs of the $2n+1$ arcs, and the writhe is the sum of their signs.

These signs depend on the arcs' directions and heights. Without loss of generality, number the arcs clockwise according to their directions in the plane of the diagram. These directions are equally placed with angles of $2\pi/(2n+1)$ in between. Given a pair of arcs $i,j$, such that $(j-i) \bmod(2n+1) \in \{1,\dots,n\}$, the sign of their crossing is positive if $\pi(i) < \pi(j)$ and negative otherwise. Therefore, the writhe of the petal diagram defined by $\pi$ coincides with the writhe of the permutation $\pi$, as defined in the introduction to this paper.

\subsection{Variations}

We proceed with a few simple manipulations to petal diagrams, that will come useful later.

\begin{figure}[htb]
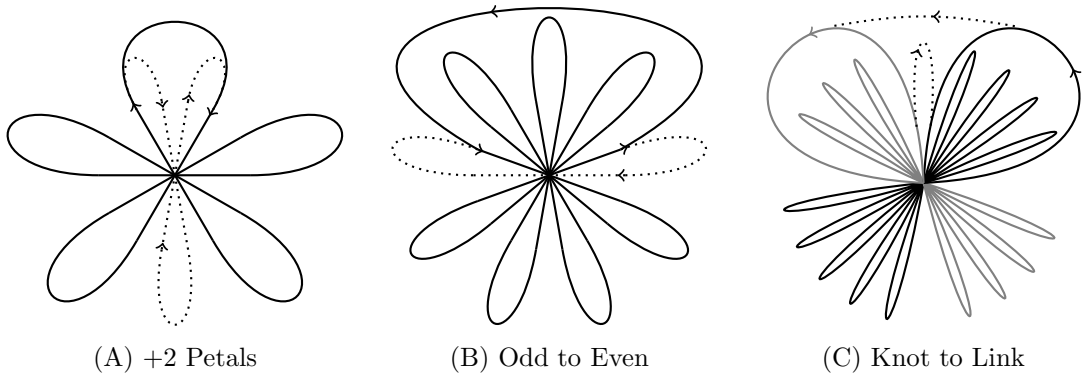

\begin{center}
\begin{tabular}{ccc}
\tikz[thick,scale=1]{
\foreach \angle/\shmangle in {0/30,60/120,150/180,210/240,300/330} {
\draw (\angle:1) .. controls +(\angle:1.8) and +(\shmangle:1.8) .. (\shmangle:1); \draw (\angle:1) -- (\angle:-1);}
\draw[postaction={decorate},dotted,decoration={markings,mark=at position 0 with {\arrow{<}}}] (80:1) -- (80:-1); 
\draw[postaction={decorate},dotted,decoration={markings,mark=at position 1 with {\arrow{<}}}] (100:-1) -- (100:1);
\draw[postaction={decorate},dotted,decoration={markings,mark=at position 0 with {\arrow{<}}}] (60:1) .. controls +(60:0.9) and +(80:0.9) .. (80:1);
\draw[postaction={decorate},dotted,decoration={markings,mark=at position 0.99 with {\arrow{<}}}] (100:1) .. controls +(100:0.9) and +(120:0.9) .. (120:1);
\draw[postaction={decorate},dotted,decoration={markings,mark=at position 0 with {\arrow{<}}}] (260:1) .. controls +(260:1.35) and +(280:1.35) .. (280:1);
\pgfresetboundingbox \clip (-2.25,-2.1) rectangle (2.25,2.4);} &
\tikz[thick,scale=1]{
\foreach \angle in {40, 80, 120, 200, 240, 280, 320} \draw (\angle:1) .. controls +(\angle:1.5) and +(\angle+20:1.5) .. (\angle+20:1);
\draw[postaction={decorate},dotted,decoration={markings,mark=at position 0.995 with {\arrow{<}}}] (0:1) .. controls +(0:1.5) and +(20:1.5) .. (20:1);
\draw[postaction={decorate},dotted,decoration={markings,mark=at position 0.005 with {\arrow{<}}}] (160:1) .. controls +(160:1.5) and +(180:1.5) .. (180:1);
\foreach \angle in {40, 80, ..., 320} \draw (\angle:1) -- (\angle:-1);
\draw[dotted,postaction={decorate},dotted,decoration={markings,mark=at position 0.05 with {\arrow{>}}}] (0:1) -- (0:-1);
\draw[postaction={decorate},decoration={markings,mark=at position 0.6 with {\arrow{>}}}] (20:1) .. controls +(30:5) and +(150:5) .. (160:1);
\pgfresetboundingbox \clip (-2.25,-2.1) rectangle (2.25,2.4);} &
\tikz[thick,decoration={markings,mark=at position 0.45 with {\arrow{>}}}]{
\foreach \angle/\color in {18/black,34/black,50/black,66/black,116/gray,132/gray,148/gray,
188/black,204/black,220/black,236/black,252/black,286/gray,302/gray,318/gray,334/gray} 
\draw[color=\color] (0,0) .. controls +(\angle-2:2.5) and +(\angle+8:2.5) .. (0,0);
\foreach \angle/\color in {3/black,93/gray} \draw[postaction={decorate},color=\color] (0,0) .. controls +(\angle:4.5) and +(\angle+84:4.5) .. (0,0);
\draw[postaction={decorate},color=black,dotted] (97:0.75) .. controls +(96:1.5) and +(84:1.5) .. (83:0.75);
\draw[postaction={decorate},color=black,dotted] (60:2.4) .. controls +(170:1) and +(10:1) .. (120:2.4);
\pgfresetboundingbox \clip (-2.25,-2.0) rectangle (2.25,2.5);} \\
(A) +2 Petals & (B) Odd to Even & (C) Knot to Link
\end{tabular}
\end{center}
\caption{Three manipulations of petal diagrams.}
\label{fig-moves}
\end{figure}

\textbf{(A)} 
We add two petals to a given petal diagram, as in Figure~\ref{fig-moves}A. If the two new arcs are assigned with two consecutive heights, then the unframed knot type is unchanged. Indeed, one can pull the whole supplement back to the original knot. Moreover, since the new arcs have adjacent heights and opposite directions, their crossings with any third arc have opposite signs, with no contribution to the writhe. The only change to the writhe comes from the crossing between the two new arcs, which is $\pm 1$, depending on the choice of heights.

By iterating this move, we can set the writhe at will. As a framed knot is determined by its unframed knot type and framing number, we conclude that the Petaluma model is universal for framed knot, in the following sense.
\begin{prop}\label{universal}
For every framed knot $K$ there exists a number $N \in \N$, such that for every $n \in \{N,N+2,N+4,\dots\}$, $K$ can be realized by a petal diagram with $2n+1$ petals.
\end{prop}

\textbf{(B)}
One can take the highest arc in a $(2n+1)$-petal diagram, and continuously move it outside, as illustrated in Figure~\ref{fig-moves}B for $n=4$. This operation preserves both the knot type and the writhe. Indeed, since the highest arc is moved away, exactly $n$ of the $2n$ eliminated crossings are positive, and $n$ are negative. Hence the new, $2n$-petal diagram in Figure~\ref{fig-moves}B yields the same framed knot. In conclusion, a random permutation $\pi \in S_{2n}$ yields a random framed knot, equivalent in distribution to the one obtained from $(2n+1)$-petal diagrams. 

\textbf{(C)}
Two petal diagrams, with $2l$ and $2r$ petals, can be superpositioned to create a $2$-component link diagram with $2(l+r)$ petals and a single multi-crossing point. See Figure~\ref{fig-moves}C for an example with $10$ and $8$ petals, colored black and gray respectively. A permutation $\pi \in S_{2(l+r)}$ can determine the heights of all arcs and yield a link. This representation is universal for links~\cite{adams2015knot}. Similar to framed knots, by taking $\pi$ uniformly at random and using the blackboard framing, we obtain a random \emph{framed link}. This model extends further to more than two components.

By reconnecting the two large petals of a link diagram, as do the dotted arcs in Figure~\ref{fig-moves}C, one obtains a petal diagram of a knot, without changing any crossing or introducing new ones. Comparing the sum of crossing signs yields the relation 
$$ w(K) = w(K_L) + w(K_R) + 2\cdot lk(K_L,K_R) $$
where $K_L$ and $K_R$ are the two components in the link diagram, and $K$ is the combined knot diagram.

\section{Properties and Connections}\label{section-properties}

In the following sections, we study several features of the writhe, and how they are reflected by different viewpoints, pointing at the relation of the writhe to previously studied notions.

\subsection{Alternating Inversion Numbers}\label{section-properties-1}

We give an equivalent expression for the writhe of a permutation. First, consider the following three statistics of a permutation $\sigma \in S_N$.
\begin{align*}
\iota(\sigma) \;&=\; \sum\limits_{0 \leq y < x < N} \;\text{sign}(\sigma(x)-\sigma(y)) \\
\Hiota(\sigma) \;&=\; \sum\limits_{0 \leq y < x < N} (-1)^{y}\;\text{sign}(\sigma(x)-\sigma(y)) \\
\HHiota(\sigma) \;&=\; \sum\limits_{0 \leq y < x < N} (-1)^{y+x}\;\text{sign}(\sigma(x)-\sigma(y))
\end{align*} 
The first statistic, $\iota(\sigma)$, is the classic and well-studied \emph{inversion number}~(e.g. \cite[page 36]{stanley1999enumerative}), up to the affine translation $\iota(\sigma) = 2\text{inv}(\sigma) - \nobreak\tbinom{N}{2}$. The \emph{alternating inversion number}, $\Hiota(\sigma)$, was defined and studied in~\cite{chebikin2008variations}. Along these lines, we refer to $\HHiota(\sigma)$ as the \emph{bi-alternating inversion number} of $\sigma$.

Let $\sigma_N \in S_N$ be drawn uniformly at random. By grouping the terms according to $y$, one can show that each of $\iota(\sigma_N)$ and $\Hiota(\sigma_N)$ follows the \emph{Mahonian} distribution, namely
$$ \frac{\iota(\sigma_N) + \tbinom{N}{2}}{2} \;=\; \text{inv}(\sigma_N) \;\sim\; U_1 + U_2 + \dots + U_N $$
where each $U_i$ is uniform in $\{0,1,\dots,i-1\}$, and all $\{U_i\}$ are independent. By the central limit theorem, as $N \to \infty$, the normalized inversion number $3\iota(\sigma_N) / N^{3/2}$ converges in distribution to a standard Gaussian~\cite[page 257]{feller1968introduction}. Somewhat surprisingly, Corollary~\ref{limdist}, together with the following lemma, show that the case of $\HHiota(\sigma_N)$ is substantially different.

\begin{lemma}\label{bialternating}
The following three variables follow the same distribution law.
\begin{enumerate}
\item $w(\pi)$ for uniform $\pi \in S_{2n+1}$
\item $\HHiota(\sigma)$ for uniform $\sigma \in S_{2n+1}$
\item $\HHiota(\sigma)$ for uniform $\sigma \in S_{2n}$
\end{enumerate}
\end{lemma}

\begin{proof}
Let $\sigma \in S_{2n+1}$, and define $\tau \in S_{2n+1}$ by $\tau(x) = 2x \mod (2n+1)$. In order to establish $w(\sigma \circ \tau) = \HHiota(\sigma)$, we compare the following sums.
$$ \sum\limits_{i=0}^{2n} \sum\limits_{j=1}^{n} \text{sign} \left(\sigma(2(i+j))-\sigma(2i)\right) \;=\; \sum\limits_{x=1}^{2n} \sum\limits_{y=0}^{x-1} (-1)^{y+x} \;\text{sign}(\sigma(x)-\sigma(y))$$
First consider $x$ and $y$ of equal parity. These terms correspond to $j \equiv (x-y)/2$ and $i \equiv y/2$, where all computations are in $\Z_{2n+1}$. Note that $j\in \{1,\dots,n\}$ as required. Otherwise, $x$ and $y$ have different parities, and we have $(-1)$ in the term on the right. In this case, the corresponding term on the left is given by $j \equiv (y-x)/2$ and $i \equiv x/2$, and the change of sign cancels with the interchange of the minuend and the subtrahend. This proves the equivalence of (1) and (2).

Let $\rho(x) = x + 1\mod (2n+1)$. By the symmetric definition of the writhe, it is invariant to rotations, i.e., $w(\pi \circ \rho) = w(\pi)$. Hence, so is the bi-alternating inversion number, $\HHiota(\sigma \circ \rho) = w(\sigma \circ \rho \circ \tau) = w\left(\sigma \circ \rho \circ \tau \circ \rho^n\right) = w(\sigma \circ \tau) = \HHiota(\sigma)$.

Let $\sigma \in S_{2n+1}$ be drawn uniformly. By rotation invariance, we can assume $\sigma(2n)=2n$, without modifying the distribution of $\HHiota(\sigma)$. Consider the terms with $x=2n$ in the definition of $\HHiota(\sigma)$. Since for these $(-1)^x = \text{sign}(\sigma(x)-\sigma(y)) = 1$, they simplify to $(-1)^y$, and their total contribution is $0$. In other words, we can rotate and delete $2n$ from each permutation, and the bi-alternating inversion number is preserved. Hence (2) and (3) are equivalent as well.

It is perhaps worth remarking that the bi-alternating inversion number $\HHiota$ is no longer invariant to rotations in the even case.
\end{proof}

\begin{rmk*}
Lemma~\ref{bialternating} has a simple interpretation in terms of the petal diagram. The equality $w(\sigma \circ \tau) = \HHiota(\sigma)$ means that the two functions go over the $2n+1$ arcs in different orders. Instead of sorting the oriented arcs by their angles as does the writhe, $\HHiota(\sigma)$ goes over the arcs as they occur in a journey along the planar curve. For example, Figure~\ref{fig-petals}A, demonstrates $w(0413265)=\HHiota(0246153)=-1$. The rotation invariance reflects the arbitrariness of choice of starting point. The transition from $S_{2n+1}$ to $S_{2n}$ corresponds to the act illustrated in Figure~\ref{fig-moves}B.
\end{rmk*}

\subsection{Efficient Computation}
If we take a naive approach to computing the writhe of $\pi \in S_{2n+1}$, the required time would be quadratic. It turns out that some re-arrangement can help.

\begin{prop}\label{compwrithe}
Given $\pi \in S_{2n+1}$, the writhe $w(\pi)$ can be computed in time $O(n\log n)$.
\end{prop}

\begin{proof}
We can work with $\HHiota(\sigma)$, where $\sigma \in S_{2n}$ is obtained from $\pi$ in time $O(n)$, as described in the proof of Lemma~\ref{bialternating}.

Let $\sigma \in S_{2(l+r)}$, and denote $L = \{0,\dots,2l-1\}$ and $R = \{2l,\dots,2(l+r)-1\}$. We define $\sigma_L \in S_{2l}$ and $\sigma_R \in S_{2r}$, as the induced permutations on the first $2l$ and last $2r$ entries of $\sigma$, respectively. They can easily be computed via $\sigma^{-1}$ in linear time. 

Similar to the linking number in~\cite{petaluma}, we consider
$$ lk_{LR}(\sigma) \;:=\; \frac12 \sum\limits_{y \in L} \;\sum\limits_{x \in R} \;(-1)^{y+x}\;\text{sign}(\sigma(x)-\sigma(y)) \;=\; \frac12 \sum\limits_{t=1}^{2(l+r)} \; y(t)(x(t)-x(t-1)) \;,$$
where
\begin{align*}
x(t)&\;:=\;\sum\limits_{i \in R, \sigma(i) < t}\;(-1)^i &
y(t)&\;:=\;\sum\limits_{i \in L, \sigma(i) < t}\;(-1)^i
\end{align*}
The functions $x(t),y(t)$ are similarly computable via $\sigma^{-1}$ in linear time, and hence so is $lk_{LR}(\sigma)$. 

Finally, note that
$$ \HHiota(\sigma) \;=\; \HHiota(\sigma_L) + \HHiota(\sigma_R) + 2\cdot lk_{LR}(\sigma) \;. $$
Therefore, we can set $l = \left\lceil n/2\right\rceil$ and $r = \left\lfloor n/2\right\rfloor$, and proceed recursively by divide and conquer. After $\log_2 n$ divisions, we reach the base cases $\HHiota(01)=-1$ and $\HHiota(10)=+1$. Hence the total running time is $O(n \log n)$.
\end{proof}

We remark that the above algorithm is inspired by the analogous problem on the writhe of petal diagrams. The division steps correspond to the manipulation presented in Figure~\ref{fig-moves}C.

\subsection{Extreme Values}

How large can the writhe of a permutation $\pi \in S_{2n+1}$ get? By a direct computation, the two simple permutations $\pi(x)=x$ and $\pi(x)=2n-x$ yield $w(\pi) = \pm n^2$. We note that these two permutations correspond to framings of the $(n+1,\pm n)$-torus knots. The following proposition shows that these two examples are the extremes.

\begin{prop}
The writhe $w(\pi)$ of a permutation $\pi \in S_{2n+1}$ is between $\pm n^2$.
\end{prop}

\begin{proof}
First consider all $2n$ pairs $x,y \in \Z_{2n+1}$ where $\pi(x)=0$. Clearly, each $y \in \{x+1,\dots,x+n\}$ contributes $+1$, and the other $n$ contribute $-1$ each, so they cancel out. The contribution of the next $2n-1$ pairs $x,y$ where $\pi(y) > \pi(x) = 1$ is $\pm 1$, depending on the distribution of such~$y$ between the two sides $\{x+1,\dots,x+n\}$ and $\{x-n,\dots,x-1\}$. Similarly, the contribution of the $2n-2$ pairs where $\pi(y) > \pi(x) = 2$ is between $\pm 2$, and so on. For pairs with $\pi(y) > \pi(x) \geq n$, we simply use their number as a bound.

We verify that these considerations yield the desired bound. Write
$$ w(\pi) \;=\; \sum\limits_{x=0}^{2n} w_x(\pi) \;\;\;\;\;\;\;\text{where}\;\; w_x(\pi) \;:=\; \sum\limits_{\pi(y) > \pi(x)} \begin{cases}  
+1 \;&\; y \in \{x + 1, \dots ,x + n\} \\ -1 \;&\; y \in \{x - n, \dots ,x - 1\} \end{cases} $$ 
Then
$$ |w(\pi)| \;\leq\; \sum\limits_{x=0}^{2n} \left|w_x(\pi)\right| \;\leq\; \sum\limits_{x=0}^{2n}\min(\pi(x),2n-\pi(x)) \;=\; n^2 $$
In fact, it follows form the argument, that equality is obtained only by the two examples mentioned above, up to rotation. 
\end{proof}

The symmetric group $S_{2n+1}$ is generated by \emph{adjacent circular transpositions} $\tau_x = (x,x+1) \in S_{2n+1}$ where $x \in \Z_{2n+1}$. Note that $\tau_{2n} = (2n,0)$ is also considered an adjacent transposition. It is easy to verify that $w(\pi \circ \tau_x) - w(\pi) = \pm 2$, depending on the circular orientation of $\pi(x)$, $\pi(x+1)$ and $\pi(x+n+1)$. Consequently, the permutations $\pi(x)=x$ and $\pi(x)=2n-x$ are at least $n^2$ transpositions apart from each other. 

This is relevant to \emph{circular bubble sort}. How many adjacent circular transposition are needed rearrange a circular array of $N$ numbers in a given way? An optimal algorithm with running time~$O(N^2)$ was given by Jerrum~\cite{jerrum1985complexity}. In the case of sorting the permutation $\pi(x)=x+\left\lfloor {N}/{2}\right\rfloor$, at least $\left\lfloor {N^2}/{4}\right\rfloor$ swaps are required due to the distance of each number from its destination. This has recently been proved to be the worst case~\cite{van2014upper}. By the above discussion, writhe considerations easily yield a very close bound of $(N-1)^2/4$ transpositions, for $\pi(x)=N-x$. This case, unlike the former, has the advantage of staying relevant in the relaxed setting of sorting up to rotations.

\subsection{Circular Rank Correlation}\label{stat}

Recall the left rotation invariance of the writhe, $w(\pi) = w(\rho^k \circ \pi)$ where $\rho(x)=x+1$ for $x \in \Z_{2n+1}$. Averaging $w(\rho^k \circ \pi)$ over $0 \leq k \leq 2n$ one obtains
$$ w(\pi) \;=\; \sum\limits_{0 \leq x < y \leq 2n} \alpha_n(y-x) \cdot \beta_n(\pi(y)-\pi(x)) $$
where $\alpha_n,\beta_n : \Z_{2n+1} \setminus \{0\} \to [-1,1]$ are
$$ \alpha_n(d) \;:=\; \begin{cases} +1 & d \in \{+1,\dots,+n\} \\ -1 & d \in \{-n,\dots,-1\} \end{cases} \;\;\;\;\;\;\;\;\;\; 
\beta_n(d) \;:=\; \begin{cases} +1-\tfrac{2|d|}{2n+1} & d \in \{+1,\dots,+n\} \\ -1+\tfrac{2|d|}{2n+1} & d \in \{-n,\dots,-1\} \end{cases} $$
In words, $\alpha_n(y-x)$ only tells from which side $y$ is closest to $x$, while $\beta_n(y-x)$ is linearly dependent on their distance, giving larger weight to close pairs.

Put this way, the writhe assumes a statistical interpretation as a nonparametric statistic for circular correlation, of a certain type that we turn to describe. See~\cite{fisher1995statistical,mardia1975statistics} for overview of circular data analysis. 

Suppose that $(\theta_1,\phi_1),\dots,(\theta_N,\phi_N)$ are iid samples from some unknown distribution on the torus $S^1 \times S^1$. Without parametric assumptions on the distribution, we rely only on the ordering of~$\theta_i$ and~$\phi_i$ around the circle. Thus let the \emph{circular ranks} $r_1,\dots,r_N \in \Z_N$ and $s_1,\dots,s_N \in \Z_N$ have respectively the same circular orderings as $\theta_1,\dots,\theta_N$ and $\phi_1,\dots,\phi_N$. Suppose that we want to assess how well the~$\phi_i$ and the~$\theta_i$ tend to be related by a one-to-one continous map $S^1 \to S^1$. The case of complete dependance corresponds to $r_i = \pm s_i+k$, while if~$\theta_i$ and~$\phi_i$ are independent then $r_i = \pi(s_i)$ where $\pi \in S_N$ uniformly.

Our treatment is analogous to the case of linear rank correlation, with samples $\in \R^2$, where the common nonparametric statistics are Kendall's tau and Spearman's rho. A simple general form of linear correlation coefficient was given by Daniels~\cite{daniels1944relation}. In a similar spirit, here it is useful to consider the following general form of circular rank correlation.
$$ R_{fg} \;=\; \sum_{1\leq i<j \leq N} f\left(\frac{r_j-r_i}{N}\right) \cdot g\left(\frac{s_j-s_i}{N}\right) $$ 
where $f$ and $g$ are odd functions with period $1$. Some special cases are:
\begin{enumerate}
\item If $N$ is odd and the ranks are related by $s_i = \pi(r_i)$ then the writhe $w(\pi)$ is exactly $R_{\alpha\beta}$ where $\alpha(t) = \text{sign}(t) $ and $\beta(t) = \text{sign}(t)(1-2|t|)$ for $t \in (-\tfrac12,\tfrac12)$. 
\item Fisher and Lee's statistic $\Delta_N$ records the tendency of clockwise triples $r_i,r_j,r_k$ to correspond to clockwise triples $s_i,s_j,s_k$. It is equivalent to~$R_{\beta\beta}$.~\cite{fisher1982nonparametric,fisher1995statistical} 
\item An alternative measure by Fisher and Lee, following Mardia, is given by $\Pi_N = R_{\gamma\gamma}$ where $\gamma(t) = \sin(2 \pi t)$.~\cite{fisher1995statistical,mardia1975statistics}
\end{enumerate}

Of course, the choice of statistical test depends on the assumptions on the model of association between $\theta_i$ and~$\phi_i$. In some applications there is no reason to assume symmetry between the two variables, e.g., due to different nature of measurement errors. In such cases, it would be interesting if an asymmetric measure such as the writhe could fit in. We are not aware of any asymmetric nonparametric statistic previously considered for circular rank correlation.

For later reference, we cite the limit theorem for Fisher and Lee's test~\cite{fisher1982nonparametric}: Under the hypothesis that the two variable $\theta_i,\phi_i$ are independent, the normalized distribution of $\Delta_N$ converges to a non-Gaussian limit distribution, given by  $\frac{1}{N}R_{\beta\beta} \to \tfrac{1}{\pi^2}\sum_{n=1}^{\infty}\sum_{m=1}^{\infty}\frac{1}{nm}L_{nm}$, where $L_{nm}$ are iid Laplace random variables, with probability density~$f_L(x) = \frac12 e^{|x|}$.

\section{Moments of Writhe}\label{section-moments}

Here and throughout, the random variable $w$ is the writhe of a random permutation, uniformly sampled from $S_{2n+1}$, or equivalently, the framing number of a random framed knot. In this section we prove Theorem~\ref{moments}, finding the asymptotic behavior of the moments of $w$. The analysis elaborates on tools used in~\cite{petaluma} for finding the limiting moments of other knot invariants. 

\begin{rmk*}
Recall the Gauss diagram formula of the writhe, given in Section~\ref{section-knots}. By feeding it into the computer program that we used there~\cite{petaluma,supplementarymaterial}, we can automatically prove $E[w^k]=O(n^k)$. This is the expected behavior of a first order invariant, cf. the main Conjecture in~\cite{petaluma}. But here we are interested also in the constant coefficient for each $k$. 
\end{rmk*}

\subsection{Never a Dull Moment}\label{dull}

We start by deriving various expressions for the moments of $w$. By the reflection symmetry $w(f)=-w(-f)$, odd moments of the writhe vanish. In the following we assume $k$ is even. 

Thanks to Lemma~\ref{bialternating}, we study the distribution of the writhe via the bi-alternating inversion number.
$$ E\left[w^k\right] \;=\; E_{\pi}\left[ \left(\sum_{0 \leq y < x \leq 2n} (-1)^{x+y} \;\text{sign}(\pi(x)-\pi(y))\right)^k \right] $$
By the linearity of the expectation,
\begin{equation}
E\left[w^k\right] \;=\; \sum_{\mathbf{y}<\mathbf{x}} \; E_{\pi}\left[\prod_{i=1}^k \text{sign}(\pi(x_i)-\pi(y_i))\right] \;\cdot\; \prod_{i=1}^k (-1)^{x_i+y_i} \label{exy}
\end{equation}
where the notation $\mathbf{y}<\mathbf{x}$ for $\mathbf{x},\mathbf{y} \in \{0,\dots,2n\}^k$ stands for the conjunction of $y_i < x_i$ for every $i \in \{1,\dots,k\}$. This sum needs some rearrangement.

First, we encode the ordering and the equalities between the entries of $\mathbf{x}$ and $\mathbf{y}$ by a directed graph, $G = (V,E)$. Both the edges and the vertices of $G$ are labeled:
\begin{align*}
E(G) \;&=\; (e_1,\dots,e_k) \;=\; \left((v_{11},v_{12}),\dots,(v_{k1},v_{k2})\right) \\
V(G) \;&=\; e_1 \cup \dots \cup e_k \;=\; \{1,\dots,v\}
\end{align*}
For each $\mathbf{x},\mathbf{y} \in \{0,\dots,2n\}^k$, the graph $G = G_{\mathbf{xy}}$ is defined so that the map sending $y_i$ to $v_{i1}$ and $x_i$ to $v_{i2}$ is an order preserving bijection. For example, $x_i<x_j$ iff $v_{i2}<v_{j2}$, and $x_i=y_j$ iff $v_{i2}=v_{j1}$,~etc. Note that $G$ is acyclic, with $v \leq 2k$ vertices, none of which are isolated, and possibly with parallel edges. 

By the condition $\mathbf{y}<\mathbf{x}$ on each term in the sum, the vertex labeling of its graph is monotone with respect to the directed edges, i.e., $v_{i1} < v_{i2}$ for all $i$. We denote the collection of all such graphs by~$\mathcal{G}_k$.

Let $G \in \mathcal{G}_k$. We define the \emph{parity} function $\varepsilon_{\mathbf{xy}} : V(G_{\mathbf{xy}}) \to \{\pm 1\}$, such that $\varepsilon_{\mathbf{xy}}(v_{i1}) = (-1)^{y_i}$ and $\varepsilon_{\mathbf{xy}}(v_{i2}) = (-1)^{x_i}$. We also use vector notation, $\varepsilon := (\varepsilon(1),\dots,\varepsilon(v)) \in \{\pm 1\}^v$. Here and below $v:=|V(G)|$. 

We rewrite the above sum (\ref{exy}), dividing into cases according to graphs and parity.
\begin{equation}
E[w^k] \;=\; \sum_{G \in \mathcal{G}_k} \; \sum_{\mathbf{\varepsilon} \in \{\pm\}^v} C(G,\varepsilon) \; A(G) \; T(G,\varepsilon) \label{cat}
\end{equation}
where
\begin{align*}
C(G,\varepsilon) \;&:=\; \#\left\{ \mathbf{x},\mathbf{y} \in \{0,\dots,2n\}^k \;:\; G_{\mathbf{xy}}=G,\;\varepsilon_{\mathbf{xy}}=\varepsilon\right\} \\
T(G,\varepsilon) \;&:=\; \prod_{i=1}^k \varepsilon(v_{i1})\varepsilon(v_{i2})
\end{align*}
and 
$$ A(G) \;=\; E_{\pi}\left[\prod_{i=1}^k \text{sign}(\pi(x_i)-\pi(y_i)) \right] $$
for any $\mathbf{x},\mathbf{y}$ such that $G_{\mathbf{xy}}=G$. 

Of course, we have to verify that $A(G)$ is well-defined, i.e., the right-hand side does not depend on $\mathbf{x},\mathbf{y}$. Indeed, it only depends on the ordering of the $|V(G)|$ values that $\pi$ assigns to the entries in $\mathbf{x}$ and $\mathbf{y}$, and since $\pi$ is uniform the $|V(G)|!$ orders are equally likely. In general, we define the \emph{average sign} of a $v$-vertex directed graph $G$,
$$ A(G) \;:=\; \frac{1}{v!} \;\sum_{\sigma \in S_v} \;\prod\limits_{(i,j) \in E(G)} \text{sign}(\sigma(j)-\sigma(i)) $$ 
This parameter was considered by Kalai~\cite{kalai2002fourier} in the study of Fourier expansions for social choice functions.

To proceed with the analysis of $E[w^k]$ we borrow a definition and a lemma from~\cite{petaluma}. 

\begin{lemma}[\cite{petaluma}, Lemma 11] \label{runs}
$\;$ \\
Let $\varepsilon \in \{\pm1\}^v$. Denote by $z(\varepsilon)$ be the number of runs of $+$'s in $\varepsilon$. For example, 
$$z(+++--)=1,\;\;\;z(+-+)=2$$
Equivalently,
$$ z(\varepsilon_1,\dots,\varepsilon_v) \;=\; \frac{(v+1) +  \varepsilon_1 - \varepsilon_1\varepsilon_2 - \varepsilon_2\varepsilon_3 - \cdots - \varepsilon_{v-1}\varepsilon_v + \varepsilon_v}{4} $$
Then 
$$ \#\{0 \leq t_1 < \cdots < t_v \leq 2n \;:\; \forall i \; (-1)^{t_i} = \varepsilon_i \} \;=\; \binom{n + z(\mathbf{\varepsilon})}{v}
\;=\; \sum_{r=0}^{v} \binom{n}{v-r} \binom{z(\mathbf{\varepsilon})}{r}$$
\end{lemma}

The proof of Lemma~\ref{runs} is by straightforward bijections. See~\cite{petaluma} for more details. Now we can substitute 
$$C(G,\varepsilon) = \binom{n + z(\mathbf{\varepsilon})}{|V(G)|} \;.$$
Note that $A(G)$ and $T(G,\varepsilon)$ do not depend on the underlying parameter $n$. Hence (\ref{cat}) expresses $E[w^k]$ as a polynomial in $n$, computable in finite time for each $k$. For example, for $k=2,4$ we evaluate all terms and obtain the following.
\begin{prop}\label{moments24}
Let $w$ be the writhe of $\pi \in S_{2n+1}$, sampled uniformly at random. Then
$$ E\left[w^2\right] \;=\; \frac{2n^2+n}{3} \;\;\;\;\;\;\;\;\;\; E\left[w^4\right] \;=\; \frac{76n^4+44n^3-49n^2-26n}{45} $$
\end{prop}

The treatment of the leading term of the $k$th moment is similar to~\cite{petaluma}, with partitions or patterns replaced by graphs. Recall that the functions $\chi(I)(\varepsilon_1,\dots,\varepsilon_v) \;=\; \prod_{i \in I}\varepsilon_i$, where $I \subseteq \{1,\dots,v\}$, constitute an orthogonal basis of the $2^v$-dimensional space of functions $\{\pm 1\}^v \to \R$, and $\left\langle \chi(I),\chi(J) \right\rangle = \sum_{\varepsilon} \chi(I)(\varepsilon) \chi(J)(\varepsilon) = \nobreak 2^v\delta_{IJ}$. The \emph{degree} of a function $f:\{\pm 1\}^v \to \R$ is defined as the largest $|I|$ in its representation, $f = \sum_I\phi_I\chi(I)$, where $\phi_I \neq 0$. By Lemma~\ref{runs}, $\deg z = 2$ as a function of $\varepsilon \in \{\pm 1\}^v$. For $r \leq v/2$, one can verify that $\deg \tbinom{z}{r} = 2r$. In these terms, (\ref{cat}) becomes
\begin{equation}\label{act}
E[w^k] \;=\; \sum_{G \in \mathcal{G}_k} A(G) \left\langle C(G,\varepsilon),  T(G,\varepsilon) \right\rangle \;=\; \sum_{G \in \mathcal{G}_k} A(G) \; \sum_{r=0}^{v} \;\binom{n}{v-r} \left\langle \binom{z}{r}, \chi(I_G) \right\rangle 
\end{equation}
where $I_G$ is the set of vertices in $G$ with odd total degree. Note that $|I_G|$ is even, and $|I_G| \geq 2(v-k)$ with equality iff all vertices have total degree $1$ or $2$. The terms $\tbinom{n}{v-r}\langle\tbinom{z}{r},\chi(I_G)\rangle$ divide into cases:
\begin{itemize}
\item If $r > v - k$ then $\tbinom{n}{v-r} = O(n^{k-1})$.
\item If $r < v - k$ then $\langle\tbinom{z}{r},\chi(I_G)\rangle = 0$, since $\deg\tbinom{z}{r} = 2r < 2v-2k \leq |I_G|$.
\item If $r = v - k < |I_G|/2$ then similarly the term vanishes by orthogonality.
\item If $r = v - k = |I_G|/2$, we have $\tbinom{n}{k} \langle\tbinom{z}{|I_G|/2},\chi(I_G)\rangle$.
\end{itemize}
Note that this means $E[w^k] = O(n^k)$.  

We characterize the graphs that make nonzero contribution to the sum in~(\ref{act}). Which components in the expansion of $\tbinom{z}{r}$ according to the basis $\{\chi(I)\}_I$ have degree $2r$? By examining the expression for $z(\varepsilon)$ in Lemma~\ref{runs}, every $I$ that is a union of $r$ disjoint pairs of consecutive numbers appears in $\tbinom{z}{r}$ with coefficient $(-\frac14)^r$. Denote by $\mathcal{G}^*_k$ the graphs in $\mathcal{G}_k$, with $I_G$ of this form and maximal total degree $2$. Every $v$-vertex graph in $\mathcal{G}^*_k$ contributes $\tbinom{n}{k} (-\frac14)^r 2^v \;=\; \tbinom{n}{k} (-1)^v 2^{2k-v}$.

\newcommand{\arcsgraph}[2]
{\tikzstyle{every node}=[circle, draw, minimum width=2pt]
\begin{tikzpicture}[baseline=-4pt]
\foreach \x in {1,...,#1} \filldraw (\x*0.4,0) circle (1pt);
\foreach \x/\y/\z in {#2} \draw[shorten >=1pt,-,decoration={markings,mark=at position 0.65 with {\arrow{>}}},postaction={decorate}] (\x*0.4,0) .. controls +(90-\z:0.2) and +(90+\z:0.2) .. (\y*0.4,0) ;
\end{tikzpicture}}

\begin{example*}
Let $G = \left(\arcsgraph{9}{1/2/135,1/4/45,3/4/135,5/8/30,5/6/135,7/9/135}\right) \in \mathcal{G}^*_6$, where the $v=9$ vertices are labeled by $\{1,\dots,9\}$ from left to right, and the $k=6$ edge are labeled arbitrarily. Then $I_G = \{2,3,6,7,8,9\}$, which is composed of $r=3$ consecutive pairs. Its contribution to $E[w^6]$ is equal to $-2^3\tbinom{n}{6}$.
\end{example*}

Dividing (\ref{act}) by $n^k$, and neglecting lower order terms, we write the limiting moments of the normalized writhe as
\begin{equation}\label{mu1}
\mu_k \;:=\; \lim_{n \to \infty} E\left[\left(\frac{w(\pi_{2n+1})}{n}\right)^k\right] \;=\; \frac{1}{k!} \sum_{G \in \mathcal{G}_k^*} (-1)^{v(G)} 2^{2k-v(G)} A(G)
\end{equation}
Let $\mathcal{G}^2_k$ be the $2$-regular graphs in $\mathcal{G}_k$. Each $G \in \mathcal{G}^*_k$ is obtained from \emph{breaking} some $G' \in \mathcal{G}^2_k$, defined as splitting some subset of the degree-two vertices into pairs of degree-one vertices, while preserving the order of the vertex labels. We denote breaking by $G \prec G'$. 

\begin{example*}
We break $3$ vertices of the graph $G' = \left(\arcsgraph{6}{1/2/135,1/3/45,2/3/135,4/6/30,4/5/135,5/6/135}\right) \in \mathcal{G}^2_6$ to obtain $G$ from the previous example. A different labeled graph for the same breaking is $\left(\arcsgraph{9}{1/3/135,1/4/45,2/4/135,5/9/30,5/6/135,7/8/135}\right)$.
\end{example*}

Note that there are $2^k$ ways to break a $k$-vertex graph at a subset of its vertices, and each breaking of $r$ vertices corresponds to $2^r$ labeled graphs in $\mathcal{G}_k^*$, as each pair of resulting degree-one vertices can be ordered in $2$ ways. Therefore, (\ref{mu1}) becomes
\begin{equation}\label{mu2}
\mu_k \;=\; \frac{2^{k}}{k!} \sum_{G' \in \mathcal{G}_k^2} B(G') \;\;\;\;\;\;\;\;\;\; B(G') \;:=\; \sum_{G \prec G'} (-1)^{v(G)} A(G)
\end{equation}
The sum on the right is only over the $2^k$ ways to break $G'$. Instead of caring for the $2^{v-k}$ different vertex labelings of the same breaking, we have multiplied each breaking by $2^{v-k}$.

Let $\sigma \in S_k$. We define a $2$-regular directed graph $G_{\sigma}$, with vertices $\{1,\dots,k\}$ and edges $\{i,\sigma(i)\}$ which are always directed from the smaller number to the larger one. We replace the summation over $G' \in \mathcal{G}_k^2$ in~(\ref{mu2}) by summation over $\sigma \in S_k$. 
\begin{equation}\label{mu3}
\mu_k \;=\; \sum_{\sigma \in S_k} 2^{k-|\text{cycles}(\sigma)|} \; B(G_{\sigma})
\end{equation}
Indeed, we multiplied by $k!$ since each $G_{\sigma}$ corresponds to $k!$ \emph{edge} labelings of $G' \in \mathcal{G}_k^2$. We then divided by $2$ once for every cycle of $\sigma$, since both orientations of a cycle in $\sigma$ yield the same cycle in $G_{\sigma}$. Note that $G_{\sigma}$ may contain parallel edges in case of cycles of length $2$, in which case we divide by $2$ as well, but for another reason -- we have half the number of edge labelings for $G_{\sigma}$.

\subsection{Euler Numbers and Eulerian Numbers}\label{section-moments-2}

We suspend for a while our efforts to rewrite $\mu_k$, for the sake of finding $A(G)$ and $B(G')$. To that end, we'll use some classical notions in Enumerative Combinatorics. An \emph{alternating} permutation $\pi \in S_m$, is defined by the condition $\pi(1) < \pi(2) > \pi(3) < \pi(4) > \dots \;\pi(m)$. The number of alternating permutations in $S_m$, here denoted by $A_m$, is called the \emph{Euler number}~\cite{stanley1999enumerative}. Their exponential generating functions was given by Andr\'{e} as follows.
\begin{align*}	
\sec x \;&=\; 1 + \frac{1}{2!} x^2 + \frac{5}{4!}x^4 + \frac{61}{6!}x^6 + \dots + \frac{A_{m}}{m!}x^{m} + \dots && m \text{ even}\\
\tan x \;&=\; \frac{1}{1!} x + \frac{2}{3!}x^3 + \frac{16}{5!}x^5 + \frac{272}{7!}x^7 + \dots + \frac{A_m}{m!}x^{m} + \dots && m \text{ odd}
\end{align*}
A \emph{descent} in a permutation $\pi \in S_m$ is a pair $\pi(i) > \pi(i+1)$ where $i \in \{1,\dots,m-1\}$. The number of permutation in $S_m$ with $d$ descents is the \emph{Eulerian number} $A_{m,d}$. The following relation between the Euler numbers and the Eulerian numbers goes back to Euler~\cite{eulerremarques,stanley1999enumerative}.
$$ \sum_{d=0}^{m-1} (-1)^d A_{m,d} \;=\; \begin{cases} (-1)^{(m-1)/2}A_{m} & \text{if $m$ is odd} \\ 0 & \text{if $m$ is even} \end{cases}$$
Finally, the odd Euler numbers are related to the \emph{Bernoulli numbers}, via
$$ B_m \;=\; (-1)^{m/2+1} \frac{m}{2^m (2^m-1)} A_{m-1} \;\;\;\;\;\;\;\;\;\;\;\; \text{for even $m \geq 2$} $$

Let $P_m$ be an oriented path with $m$ edges, and let $C_m$ be a cyclically oriented circle with $m$ edges. The following lemma gives the average sign of oriented paths and cycles.
\begin{lemma}\label{lemmaA}
Let $A(G)$ be as in~(\ref{cat}).
$$ A(P_{m-1}) \;=\; -A(C_{m+1}) \;=\; \begin{cases} \frac{(-1)^{(m-1)/2}}{m!}A_{m} & \text{if $m$ is odd} \\ 0 & \text{if $m$ is even} \end{cases} $$
\end{lemma}

\begin{examples*}
$A\left(\arcsgraph{2}{1/2/135,2/1/315}\right) = -1$, 
~$A\left(\arcsgraph{4}{1/2/135, 2/3/135, 3/4/135, 4/1/315}\right) = - A\left(\arcsgraph{3}{1/2/115, 2/3/115}\right) = \tfrac13$,
~$A\left(\arcsgraph{5}{1/2/115, 2/3/115, 3/4/115, 4/5/115}\right) = \frac{2}{15}$.
\end{examples*}

\begin{proof}
The edges of $P_{m-1}$ are $(i,i+1)$ for $i \in \{1,\dots,m-1\}$. The product $\prod_i\text{sign}(\sigma(i+1)-\sigma(i))$ is exactly $(-1)^d$ where $d$ is the number of descents in $\sigma$. Hence the average over $\sigma$ is the alternating sum $\tfrac{1}{m!}\sum_d(-1)^d A_{m,d}$, and we use the above identity for the Eulerian numbers.

In the cycle $C_{m+1}$, consider the vertex $i = \sigma^{-1}(m+1)$ given $\sigma \in S_{m+1}$. Note that the product $\text{sign}(\sigma(i)-\sigma(i-1)) \cdot \text{sign}(\sigma(i+1)-\sigma(i))$ equals $(+1)\cdot(-1)$ regardless of the rest of $\sigma$. The remaining $m$ vertices constitute a path $P_{m-1}$ on which $\sigma$ is uniform in $S_{m}$, as before. Hence $A(C_{m+1}) = - A(P_{m-1})$.
\end{proof}

By the next set of observations, knowing the average sign $A(G)$ for oriented paths and cycles is all we need to evaluate (\ref{mu2}).

\begin{lemma}\label{lemma3}
Let $A(G)$ be the average sign of $G$ as before.
\begin{enumerate}
\item If $G_1$ is obtained from $G_2$ by flipping the direction of one edge, then $A(G_1) = -A(G_2)$.
\item If the number of edges in $G$ is odd, then $A(G) = 0$.
\item If $G$ is the disjoint union of connected components $G_1,\dots,G_l$ then $A(G) = \prod_i A(G_i)$.
\end{enumerate}
\end{lemma}

\noindent \emph{Proof.} 
All parts are fairly straightforward from the definition:
\begin{enumerate}
\item Note that $\text{sign}(\sigma(i)-\sigma(j))$ changes for exactly one edge.
\item Equivalently summing over $\overline{\sigma}(i) := \sigma(v+1-i)$ yields a factor of $(-1)^{\#\text{edges}}$. 
\item The restriction of $\sigma \in S_m$ to each connected component is independently uniform.\hfill\qed
\end{enumerate}

We turn to compute $B(G')$, the alternating sum of $A(G)$ over all the breakings $G \prec G'$, where $G'$ is a disjoint union of cycles. Note that the above-mentioned three properties naturally carry over from $A$ to $B$. Hence it suffices to know $B$ for an even oriented cycle.

\begin{lemma}\label{lemmaB} For $m \geq 2$,
$$ B(C_m) \;=\; \frac{- 2^m B_{m}}{m!}  $$
\end{lemma}

\begin{proof}
Each proper breaking of $C_m$ is an ordered partition of the cycle into $r$ oriented paths, of lengths $m_1 + \dots + m_r = m$. By Lemma~\ref{lemma3},
$$ B(C_m) \;=\; A(C_m) + \sum\limits_{r=1}^m \;(-1)^r \;\frac{m}{r} \sum\limits_{m_1 + \dots + m_r = m} \; A(P_{m_1})A(P_{m_2})\cdots A(P_{m_r}) $$
Indeed, we multiplied by $m$ since we can start the breaking at any vertex in $C_m$, and divided by $r$ to avoid over-counting, since any of the $r$ parts could serve as the starting point. 

We write a generating function.
$$ \sum\limits_{m=1}^{\infty} B(C_m) x^m \;=\; \sum\limits_{m=1}^{\infty} A(C_m) x^m + \sum\limits_{m=1}^{\infty} \;mx^m\; \sum\limits_{r=1}^m \;\frac{(-1)^r}{r} \sum\limits_{\Sigma m_i = m} A(P_{m_1})\cdots A(P_{m_r}) $$
Lemma~\ref{lemmaA} yields generating functions for $A(C_m)$ and $A(P_m)$:
$$ \tanh x = \sum\limits_{m\text{ odd}}^{\infty} \frac{(-1)^{(m-1)/2} A_m x^m}{m!} \;=\; -\sum\limits_{m=1}^{\infty} A(C_m) x^{m-1} \;=\; x + \sum\limits_{m=1}^{\infty} A(P_m) x^{m+1} $$
A sequence of standard manipulations on power series yields:
$$ -x\frac{d}{dx}\log\left(1 + \sum\limits_{m=1}^{\infty}c_m x^m\right) \;=\; \sum\limits_{m=1}^{\infty} m x^m \sum\limits_{r=1}^m\frac{(-1)^r}{r} \sum\limits_{\Sigma m_i = m} c_{m_1}c_{m_2}\cdots c_{m_r} $$  
We substitute the above, and obtain:
$$ \sum\limits_{m=1}^{\infty} B(C_m) x^m \;=\; - x\tanh x - x\frac{d}{dx}\log\left(1 + \frac{\tanh x - x}{x} \right) \;=\; 1 - x \coth x $$
This yields the desired result via another expansion~\cite[p. 42]{jeffrey2007table}: 
$$ x \coth x \;=\; 1 + \sum\limits_{m=2}^{\infty} 2^mB_mx^m/m! $$
\end{proof}

We now proceed from oriented cycles to any union of cycles. An \emph{excedance} in a permutation $\sigma \in S_k$ is a pair $i < \sigma(i)$, while a pair $i > \sigma(i)$ is called a \emph{deficiency}. The number of deficiencies is denoted by $\text{def}(\sigma)$. 

\begin{lemma}\label{lemmaBG} Let $\sigma \in S_k$, and let $G_{\sigma}$ be as defined for (\ref{mu3}).
$$ B(G_{\sigma}) \;=\; \prod_{\gamma \in \text{cycles}(\sigma)} (-1)^{\text{\normalfont{def}}(\gamma)} B(C_{|\gamma|}) $$
\end{lemma}

\begin{proof}
$B$ is multiplicative at the different cycles by Lemma~\ref{lemma3}. The cyclic graph corresponding to $\gamma \in cycles(\sigma)$ is turned into $C_{|\gamma|}$ by flipping $\text{\normalfont{def}}(\gamma)$ edges, corresponding to deficiencies.
\end{proof}

The cycles structure of a permutation $\sigma \in S_m$ induces a partition $X_1 \cup \dots \cup X_r = [m]$. Given a set $X_i$, denote by $Z(X_i)$ the collection of $(|X_i|-1)!$ cycles composed of these elements. The following lemma provides the average parity of deficiencies in a cycle.

\begin{lemma}\label{lemmaAZ}
$$ \frac{1}{|Z(X)|} \sum_{\gamma \in Z(X)} (-1)^{\text{\normalfont{def}}(\gamma)} \;=\; A(C_{|X|}) $$
\end{lemma}

\begin{proof}
By considering the cyclic notation $(1, \sigma(1), \sigma^2(1), \dots )$, one obtains equivalence between the distribution of descents in all permutations, and deficiencies in cyclic permutations. In particular, the sum of $(-1)^{\text{def}(\sigma)}$ over cyclic $\sigma \in S_m$, is given by $\pm A_{m-1}$, just as in Lemma~\ref{lemmaA}.
\end{proof}

We finally return to (\ref{mu3}) and simplify the limiting moments. By Lemmas~\ref{lemmaBG}-\ref{lemmaAZ}
\begin{align*}
\frac{\mu_k}{2^k} \;&=\; \sum_{\sigma \in S_k} \frac{B(G_{\sigma})}{2^{|\text{cycles}(\sigma)|}} \;=\; \sum\limits_{X_1,\dots,X_r} \;\prod\limits_{i=1}^{r} \;\sum\limits_{\gamma \in Z(X_i)} \frac{(-1)^{\text{\normalfont{def}}(\gamma)} B(C_{|\gamma|})}{2} \;=\; \\
\;&=\; \sum\limits_{X_1,\dots,X_r} \;\prod\limits_{i=1}^{r} \;|Z(X_i)| \cdot \frac{A(C_{|X_i|}) B(C_{|X_i|})}{2} \;=\; \sum_{\sigma \in S_k} \;\prod\limits_{\gamma \in \text{cycles}(\sigma)} \frac{A(C_{|\gamma|}) B(C_{|\gamma|})}{2}
\end{align*}
Substituting $A$ and $B$ from Lemmas~\ref{lemmaA},\ref{lemmaB} proves Theorem~\ref{moments}.

\section{Limit Distribution}\label{section-limit}

We prove Corollary~\ref{limdist}, showing that the distribution of the normalized writhe, $W_{2n+1} = w(\pi)/n$ where $\pi \in S_{2n+1}$ uniformly, weakly converges to a limit $W$. We then study some representations and properties of the limit distribution, and establish its continuity.

\subsection{Proof of Corollary~\ref{limdist}}

Partition the sum over permutations in Theorem~\ref{moments}, according to the cycle lengths.
$$ \mu_k \;=\; \sum_{\sigma \in S_k} \;\prod\limits_{\gamma \in \text{cycles}(\sigma)} \lambda_{|\gamma|} \;=\; \sum_{r=1}^{\infty}\frac{1}{r!}\;\sum_{m_1+\dots+m_r=k}\;\frac{k!}{m_1 m_2 \cdots m_r} \lambda_{m_1}\lambda_{m_2}\cdots\lambda_{m_r} $$
Indeed, every \emph{ordered} set of $r$ cycles that constitute $\sigma$ will be counted $r!$ times, hence the factor $(1/r!)$. If the cycle lengths are $m_1,\dots,m_r$, then the number of assignments of numbers $\{1,\dots,k\}$ to the elements of all cycles is $k!/m_1 m_2 \cdots m_r$, because of equivalence to rotation of each cycle. Thus the moment generating function is
$$ \mathcal{M}(z) \;=\; \sum_{k=0}^{\infty} \frac{\mu_k}{k!} z^k \;=\; 1 + \sum_{r=1}^{\infty}\frac{1}{r!} \;\prod_{j=1}^r \;\sum_{m_j=1}^{\infty} \frac{\lambda_{m_j}}{m_j} z^{m_j} \;=\; \exp\left(\sum_{m=1}^{\infty} \frac{\lambda_m}{m} z^m \right) $$
Since asymptotically $|B_m| \sim 2 \cdot m! /(2\pi)^m$ for even $m \to \infty$, we have~\cite[p. 1041]{jeffrey2007table} 
$$ \lambda_m \;=\; \frac{8^m(2^m-1)B_m^2}{2(m!)^2} \;\sim\; 2\left(\frac{2}{\pi}\right)^{2m} $$ 
Hence $\log\mathcal{M}(z)$ has radius of convergence $\pi^2/4$. This implies:
\begin{enumerate}
\item The moment generating function $\mathcal{M}(z)$ is analytic at $z=0$.
\item A limit distribution $W$ exists, uniquely defined by its moments $\mu_k$~(e.g. \cite{breiman1992probability, durrett2010probability}).
\item The tails of $W$ decay at exponential asymptotic rate, $\frac{1}{t} \log P\left[|W|>t\right] \xrightarrow{\;t \to \infty\;} -\pi^2/4$.
\end{enumerate}
The continuity of $W$ is inferred from its series representation below.
\hfill\qed

\subsection{Integral Representation}

Consider the characteristic function of~$W$: 
$$ \varphi_W(t) \;=\; \exp\left( \sum_{m=2}^{\infty} \frac{8^m(2^m-1)B_m^2}{2m(m!)^2} (it)^m \right) $$
It turns out that the infinite series can be replaced with an integral. By the expansions~\cite[p. 42]{jeffrey2007table}
$$ t \cot 2t \;=\; \sum_{m=0}^{\infty} \frac{(4i)^m B_m}{2\cdot m!}t^m \;\;\;\;\;\;\;\; \log \cosh t \;=\; \int\limits_0^t \tanh u du \;=\; \sum_{m=2}^{\infty} \frac{2^m(2^m-1)B_m}{m!}
\cdot\frac{t^m}{m} $$
we represent the characteristic function as a Hadamard product,
$$ \varphi_W(t) \;=\; \exp \left( t \cot 2t \circ \log \cosh t \right) \;\;\;\;\;\;\text{where}\;\;\;\;\;\; \left(\sum a_n t^n \right) \circ \left(\sum b_n t^n \right)\;=\; \left(\sum a_n b_n t^n \right) $$
which equals the integral
$$ (f \circ g)(t) \;=\; \int_0^1 f\left(e^{2\pi iu}\right)g\left(te^{-2\pi iu}\right)du \;=\; \frac{1}{2 \pi i} \oint_{|z|=1} f(z)g\left(\tfrac{t}{z}\right) \tfrac{dz}{z} $$
Therefore, the characteristic function of $W$ is given by the contour integral
$$ \varphi_W(t) \;=\; \exp\left( \frac{1}{2 \pi i} \oint_{|z|=1} \cot (2z) \; \log \cosh (t/z) \;dz\right) $$

\subsection{Series Representations}

We study the characteristic function of $W$ by another approach. Recall that the Bernoulli numbers appear as special values of the Riemann zeta function. For even $m \geq 2$,
$$ \zeta(m) = \sum\limits_{n=1}^{\infty}\left(\frac{1}{n}\right)^m \;=\; -\frac{(2\pi i)^m B_m}{2 \cdot m!} $$
Hence $\varphi_W(t) =$
$$ \exp\left( -\sum_{m=2}^{\infty} \frac{4^m(2^m-1)B_m t^m}{m \cdot m! \cdot \pi^m} \sum\limits_{n=1}^{\infty}\left(\frac{1}{n}\right)^m \right) \;=\;
\exp\left( -\sum\limits_{n=1}^{\infty} \sum_{m=2}^{\infty} \frac{2^m(2^m-1)B_m}{m \cdot m! } \left(\frac{2t}{\pi n}\right)^m \right) $$
By the above-mentioned expansion of $\log \cosh$,
$$ \varphi_W(t) \;=\; \exp\left( -\sum\limits_{n=1}^{\infty} \log \cosh \left(\frac{2t}{\pi n}\right) \right) \;=\; \prod_{n=1}^{\infty} \text{sech}\,\left(\frac{2t}{\pi n}\right) \;=\; \prod_{n=1}^{\infty} \varphi_A\left(\frac{4t}{\pi^2 n}\right) $$
where $A$ is a random variable with characteristic function $\varphi_A(u) = \text{sech}\,\frac{\pi u}{2}$, and density function $f_A(x) = \tfrac{1}{\pi}\,\text{sech}\,x$.
By L\'{e}vy's Continuity Theorem,
$$ W \;\sim\; \frac{4}{\pi^2} \;\sum\limits_{n=1}^{\infty} \;\frac{A_n}{n} \;\;\;\;\;\;\;\; \text{where} \;\; A_n \sim A \;\; \text{iid} \;.$$
It follows that the limit distribution $W$ is absolutely continuous with respect to the Lebesgue measure. Indeed, $A_1$ has a density function, and $W$ is obtained from $A_1$ by convolution with another distribution.

This decomposition of $W$ can be further refined. A similar manipulation on the characteristic function of $A$ yields~\cite[p. 45]{jeffrey2007table}
$$ \varphi_A\left(t\right) \;=\; \text{sech}\,\left(\frac{\pi t}{2}\right) \;=\; \prod_{k=1}^{\infty} \left(1 + \frac{t^2}{(2k-1)^2}\right)^{-1} \;=\; \prod_{k=1}^{\infty} \varphi_L\left(\frac{t}{2k-1}\right) $$
where $\varphi_L(u) = 1/(1+u^2)$ is the characteristic function of the Laplace distribution, with probability density function $f_L(x) = \tfrac12 e^{-|x|}$. Equivalent definitions are $L \sim (E_1 - E_2) \sim (Z_1Z_2 + Z_3Z_4)$ where $E_i$ and $Z_i$ are respectively iid exponential and iid normal variables. In conclusion,
$$ W \;\sim\; \frac{4}{\pi^2} \;\;\sum\limits_{\stackrel{n=1}{\;}}^{\infty} \;\; \sum\limits_{\stackrel{m=1}{odd}}^{\infty} \;\frac{L_{mn}}{mn} \;\;\;\;\;\;\;\; \text{where} \;\; L_{mn} \sim L \;\; \text{iid} \;.$$

Another decomposition is manifest in the fact that $W$ is \emph{infinitely divisible}, which means that for every $n$, there exist $X_1,\dots,X_n$ iid such that $X_1 + \cdots + X_n \sim W$. This follows since the exponential distribution has this property, which is preserved when taking independent sums or limits. See an analogous treatment of the logistic distribution in~\cite{steutel1979infinite}. 

Alternatively, we derive the L\'{e}vy--Khinchin representation,
$$ \varphi_W\left(t\right) \;=\; \exp \left( \int_{0}^{\infty} (\cos ut - 1) m_W(u) du \right) 
\;\;\;\;\;\;\;\; \text{where} \;\; m_W(u) \;=\; \frac{1}{u} \sum_{n=1}^{\infty}\frac{1}{\sinh \left(n \pi^2 u / 4\right)} \;,$$
obtained from the corresponding expression for $\varphi_A(t)$ with $m_A(u) = 1/ u \sinh u$. See~\cite{levy1951wiener}.

\subsection{Remark}

Note the similarity between the limit of the writhe and that of Fisher and Lee's statistic~$\Delta_N$ from Section~\ref{stat}. The summation of $\tfrac{1}{mn}L_{mn}$, there over $m ,n \in \N$, is here restricted to odd $m$. Further investigation of this comparison suggests a limit theorem for circular rank correlation, which we sketch below. 

Recall the gerneral form $R_{fg}$ for circular rank correlation, from Section~\ref{stat}. For both the writhe and $\Delta_N$, the normalized limit distribution is given by $\frac14 \sum_{m,n} \lambda_m \kappa_n L_{mn}$, where $\lambda_m$ and $\kappa_n$ are the respective Fourier coefficients of $f$ and $g$, i.e., $f(t) = \sum_m \lambda_m \sin(2 \pi m t)$ and $g(t) = \sum_n \kappa_n \sin(2 \pi n t)$. From the theory of $U$-statistics, we know that this is indeed the case for the \emph{parametric} statistic
$$U_{fg} \;=\; \sum_{i<j} \; f(\theta_j-\theta_i) \; g(\phi_j-\phi_i)$$ 
under the hypothesis that all $\theta_i$ and $\phi_i$ are iid uniform in~$[0,1]$. The joint distribution of the normalized ranks $\tfrac{1}{N}r_i$ and $\tfrac{1}{N}s_i$ that appear in $R_{fg}$ instead, is more complicated. However, some condition on $f$ and $g$ would ensure that $\frac{1}{N}(R_{fg}-U_{fg}) \to 0$ in probability, and hence equality of the limit laws.

\bibliographystyle{alpha}
{ \footnotesize
\bibliography{writhe}
}

\end{document}